\documentclass[11pt,a4paper]{article}
\usepackage{epsfig}
\usepackage[numbers]{natbib}
\usepackage{amsmath, amsthm, amssymb}
\usepackage{epsfig}
\usepackage[dvipdfm]{hyperref}
\usepackage[margin=3cm, top=3.5cm, bottom=3.5cm]{geometry}

\newcommand\blfootnote[1]{%
  \begingroup
  \renewcommand\thefootnote{}\footnote{#1}%
  \addtocounter{footnote}{-1}%
  \endgroup
}
\newtheorem{theorem}{Theorem}[section]

\newtheorem{proposition}{Proposition}[section]

\numberwithin{equation}{section}
\everymath{\displaystyle}
\allowdisplaybreaks

\title{Control properties for heat equation with double singular potential}
\date{}
\author{
Nikolai Kutev\thanks{Institute of Mathematics and Informatics, Bulgarian Academy of Sciences, 1113, Sofia, Bulgaria}
 \and Tsviatko Rangelov
 \footnotemark[1]
}

\begin{document}

\maketitle
\blfootnote{Corresponding author: T. Rangelov, rangelov@math.bas.bg}

\begin{abstract}
\noindent
The aim of this article is to study the noncontrollability of the heat equation with double singular potential at an interior point and on the boundary of the domain.
\end{abstract}

%\vspace{2pt}

%\noindent

%\vspace{2pt}
%\noindent
{\bf Math. Subj. Class.} 26D10, 35P15

\section{Introduction}
\label{sec1}
We suppose that $\Omega$ is a star-shaped domain with respect to a ball centered at the origin, i. e.,
\begin{equation}
\label{eq1}
\Omega=\{x\in\mathbb{R}^n, n\geq3,  |x|<\varphi(x)\}, \ \ \partial\Omega=\{x: |x|=\varphi(x)\},
\end{equation}
where $\varphi(x)$ is a positive homogeneous function of 0-th  order, $\varphi(x)\in C^{0,1}(\Omega)$.

Let us consider the singular parabolic problem
\begin{equation}
\label{eq2}
\left\{\begin{array}{ll}
&u_t-\Delta u-\mu \Psi(x) u=f(t,x), \ \  f(t,x)\in L^2((0,T)\times\Omega),
\\
&u(t,x)=0 \ \ \hbox{ for }  (t,x)\in(0,T)\times\partial\Omega,
\\
&u(0,x)=u_0(x)\ \  u_0(x)\in L^2(\Omega),
\end{array}\right.
\end{equation}
where the  potential
\begin{equation}
\label{eq3}
\Psi(x)=|x|^{-2}\left[1-|x|^{n-2}\varphi^{2-n}(x)\right]
\end{equation}
is singular at the origin of the domain $\Omega$ and on the whole boundary $\partial\Omega$.

In the pioneering paper \citet{BG84} it is proved that for $\Omega\subset\mathbb{R}^n, 0\in\Omega, n\geq3$, problem \eqref{eq2} with $\Psi(x)=|x|^{-2}$ is well posed for $\mu\leq\left(\frac{n-2}{2}\right)^2$ and has a global solution.  However, for $\mu>\left(\frac{n-2}{2}\right)^2$, $u_0>0$ and $f\geq0$, problem \eqref{eq2} is ill-posed, i.e., there is complete instantaneous blow-up, see \citet{CM99}.

The motivation for the investigations of the above problem is the applications in the quantum mechanics, for example in \citet{Ca04}, where this model is derived to analyze the confinement of neutral fermions, see also \citet{PV95}. Other applications appear in molecular physics \citet{Le67}, in quantum cosmology \citet{BE97}, electron capture problems \citet{GGMS08}, porous medium of fluid \citet{AG04}.

The results in \citet{BG84} are extended in different directions, for example for general positive singular potentials, equations with variable coefficients, the asymptotic behaviour of the solutions, semilinear equations ect., see \citet{AA98, VZ00, GGR12, JYP12, Va11}.

Most of the studies of controlability theory  are in the case of interior singularities, see \citet{Va11, Er08, VZ08}. The threshold for controllability or noncontrollability of \eqref{eq2} is the optimal constant $\left(\frac{n-2}{2}\right)^2$ for  $\Psi(x)=|x|^{-2}$ in the corresponding Hardy inequality. The boundary controllability of \eqref{eq2} in \citet{Ca12}  is proved for $\mu\leq\frac{n^2}{4}$, where  $\frac{n^2}{4}$ is the optimal constant in the Hardy inequality when the potential is singular at a boundary point, i.e., $0\in\partial\Omega$.

Finally, let us mention the result in \citet{BZ16} for the potential
$$
\Psi(x)=d^{-2}(x), d(x)=\hbox{dist}(x,\partial\Omega), \Omega\subset\mathbb{R}^n, n\geq3
$$
which is singular on the whole boundary $\partial\Omega$. The authors prove existence of a unique global weak solution of \eqref{eq2} for $\mu\leq\frac{1}{4}$ where  $\frac{1}{4}$ is the optimal Hardy constant. When $\mu>\frac{1}{4}$ then there is no control which means that the blow-up, phenomena cannot be prevented, see Theorem 5.1 in \citet{BZ16}.

In the present paper we consider the case of potential \eqref{eq3} singular at an interior point and on the whole boundary of the domain $\Omega\subset\mathbb{R}^n$, $n\geq3$. We prove existence of a global weak solution for $\mu<\left(\frac{n-2}{2}\right)^2$ and boundary noncontrollability of \eqref{eq2} for $\mu>\left(\frac{n-2}{2}\right)^2$, where $\left(\frac{n-2}{2}\right)^2$ is the optimal constant in Hardy inequality, see  \eqref{eq4} below.
\section{Preliminaries}
\label{sec2}
We recall Hardy inequality for the double singular potential \eqref{eq3}.
\begin{theorem}
\label{th1}
Suppose $\Omega\subset\mathbb{R}^n$, $n\geq3$, $0\in\Omega$ and $\Omega$ is a star-shaped domain with respect to a ball centered at the origin satisfying \eqref{eq1}. For every $u(x)\in H^1_0(\Omega)$ the inequality
\begin{equation}
\label{eq4}
\int_\Omega|\nabla u|^2dx\geq \left(\frac{n-2}{2}\right)^2\int_\Omega\frac{|u|^2}{|x|^2|1-|x|^{n-2}\varphi^{2-n}(x)|^2}dx,
\end{equation}
holds. The constant $\left(\frac{n-2}{2}\right)^2$ is optimal.
\end{theorem}
\begin{proof}
The proof follows from Theorem 1.1 in \citet{FKR13} for special choice of the parameters $\alpha=1$, $\beta=1$, $p=2$ and hence $\gamma=\frac{1}{2}$, $k=n-2$, $g(s)=\frac{1}{n-2}\left(1-s^{n-2}\right)$, $v(x)=1$, $w(x)=(n-2)|x|^{-1}\left(1-|x|\varphi^{-1}(x)\right)^{-1}$. The optimality of the constant $\left(\frac{n-2}{2}\right)^2$ is proved in Theorem 1.2 in \citet{FKR13} which means that it can not be replaced with a greater one. However, equality in \eqref{eq4} is not achieved for any $u(x)\in H^1_0(\Omega)$, except in the trivial case $u(x)=0$.
\end{proof}
For $\mu<\left(\frac{n-2}{2}\right)^2$ problem \eqref{eq2} with right hand side $f(t,x)\in L^2((0,T)\times\Omega)$ has a global solution for every $t>0$ by means of Hardy inequality \eqref{eq4}.
\begin{theorem}
\label{th11}
Suppose $\Omega=\{|x|<\varphi(x)\}\subset \mathbb{R}^n$, $n\geq3$, is a star-shaped domain with respect to a small ball centered at the origin. Then if $\mu<\left(\frac{n-2}{n}\right)^2$,  problem \eqref{eq2} with $\Psi(x)$ given by \eqref{eq3} has a global solution $u(t,x)$, such that
\begin{equation}
\label{eq400}
u(t,x)\in L^\infty\left([0,\tau),L^2(\Omega)\right)\cup L^2\left((0,\tau), W^{1,2}(\Omega)\right), \ \ \hbox{ for all } \tau>0.
\end{equation}
\end{theorem}
\begin{proof} For the readers convenience we sketch the proof.
We consider the truncated function $\Psi_N(x)=\min\{N, \Psi(x)\}, N=1,2 \ldots$. Let $u_N(t,x)$ be the solution of the truncated problem
\begin{equation}
\label{eq401}
\left\{\begin{array}{ll}
&u_{N,t}-\Delta u_N=\mu \Psi_N(x) u_N+f(t,x), \ \  t>0, x\in\Omega,
\\
&u_N(t,x)=0, \ \ \hbox{ for } t>0,  x\in\partial\Omega,
\\
&u_N(0,x)=u_0(x),\ \ \hbox{ for } x\in\Omega.
\end{array}\right.
\end{equation}
Multipying the equation in \eqref{eq401} with $u_N$ and integrating by parts we get from Hardy's inequality \eqref{eq4} the following estimates for every $T>0$,
 see Theorem 4.1 in \citet{AA98}
\begin{equation}
\label{eq402}
\begin{array}{ll}
&\int_{\Omega}|u_N(T,x)|^2dx+\int_0^T\int_{\Omega}|\nabla u_N(t,x)|^2dxdt
\\&=\int_{\Omega}u_0^2(x)dx+\frac{\mu+C}{2}\int_0^T\int_{\Omega}\Psi_N(x)u_N^2(t,x)dxdt
\\&-\frac{C-\mu}{2}\int_0^T\int_{\Omega}\Psi_N(x)u_N^2(t,x)dxdt+\int_0^T\int_{\Omega}f(t,x)u_N(t,x)dxdt
\\&\leq\int_{\Omega}u_0^2(x)dx+\frac{\mu+C}{2C}\int_0^T\int_{\Omega}|\nabla u_N(x,t)|^2dxdt
-\frac{m(C-\mu)}{2}\int_0^T\int_{\Omega}u_N^2(t,x)dxdt
\\&+\frac{m(C-\mu)}{2}\int_0^T\int_{\Omega}u_N^2(t,x)dxdt+\frac{1}{2m(C-\mu)}\int_0^T\int_{\Omega}f^2(t,x)dxdt,
\end{array}
\end{equation}
where $m=\inf_{x\in\Omega}\Psi(x)>0$ and $C=\left(\frac{n-2}{2}\right)^2$. Since $\mu <C$ we get from \eqref{eq402} the energy estimate
$$\begin{array}{ll}
&\int_{\Omega}|u_N(x,T)|^2dx+\frac{C-\mu}{2C}\int_0^T\int_{\Omega}|\nabla u_N(x,t)|^2dxdt
\\
&\leq \int_{\Omega}|u_0(x)|^2dx+\frac{1}{2m(C-\mu)}\int_0^T\int_{\Omega}f^2(t,x)dxdt.
\end{array}
$$
From the comparison principle $u_N(t,x)$ is a nondecreasing sequence of functions because $\Psi_N(x)\geq \Psi_M(x)$ for every $x\in\Omega$, $t>0$ and $N\geq M$. We can pass to the limit $N\rightarrow\infty$ by using Theorem 4.1 in \citet{BM92}. Thus the global solution $u(t,x)$ of \eqref{eq2} is defined as a limit of the solution $u_N(t,x)$ of the truncated problem \eqref{eq401} and $u(t,x)$ has the regularity properties given in \eqref{eq400}.
\end{proof}
Thus the natural question is weather  $\left(\frac{n-2}{2}\right)^2$ is the sharp constant for global existence of the solutions to \eqref{eq2}. In the present paper we give more precise answer. When $\mu>\left(\frac{n-2}{2}\right)^2$ we prove nullnoncontrollability of \eqref{eq2}, i.e., it is not possible for given $u_0(x)\in L^2(\Omega)$ one to find a control function $f(t,x)\in L^2((0,T)\times\Omega)$ localized in $\omega\subset\Omega$ such that there exists a solution of \eqref{eq2}. In this way we can not prevent the blow-up phenomena acting on a subset for $\mu>\left(\frac{n-2}{2}\right)^2$.

We recall also the classical Hardy inequality
\begin{equation}
\label{eq200}
\int_\Omega|\nabla u|^2dy\geq \left(\frac{n-2}{2}\right)^2\int_\Omega\frac{u^2}{|y|^2}dy.
\end{equation}
for every $u\in H_0^1(\Omega)$ in a bounded domain $\Omega\subset\mathbb{R}^n, n\geq3$ and optimal constant $\left(\frac{n-2}{2}\right)^2$, see \citet{HPL52}.

\section{Main result}
\label{sec3}
In Theorem \ref{th2} below we prove that problem \eqref{eq2} cannot be stabilized due to the explosive modes concentrated around the singularities when $\mu>\left(\frac{n-2}{2}\right)^2$. For this purpose, following the idea of optimal control, see \citet{Er08}, we consider for any  $u_0\in L^2(\Omega)$ the functional
$J_{u_0}(u,f)=\frac{1}{2}\int_{\Omega\times(0,T)}u^2(t,x)dxdt+\frac{1}{2}\int_0^T\|f\|_{H^{-1}(\Omega)}dt$
defined in the set
$D(u_0)=\left\{(u,f)\in L^2\left((0,T);H^1_0(\Omega)\right)\times L^2\left((0,T);H^{-1}(\Omega)\right)\right\}$
where $u(t,x)$ satisfies \eqref{eq2}. Here $f(t,x)$ is the control which is null in $\Omega\backslash\bar{\omega}$, $\omega\Subset\Omega\backslash\{0\}$.

We say that \eqref{eq2} can be stabilized if there exists a constant $C_0$ such that
$$
\inf_{(u,f)\in D(u_0)}J_{u_0}(u,f)\leq C_0\|u_0\|^2_{L^2(\Omega)}, \hbox{ for every } u_0\in L^2(\Omega).
$$
Let us consider the regularized problem
\begin{equation}
\label{eq7}
\left\{\begin{array}{ll}
&u_t-\Delta u-\mu \Psi_\varepsilon u=f(t,x), \ \ \hbox{ for } (t,x)\in(0,T)\times\Omega,
\\
&u(t,x)=0 \ \ \hbox{ for }   (t,x)\in (0,T)\times\partial\Omega,
\\
&u(0,x)=u_0(x)\ \ \hbox{ for } x\in\Omega,
\end{array}\right.
\end{equation}
where
\begin{equation}
\label{eq8}
\Psi_\varepsilon(x)=\left(|x|+\varepsilon\right)^{-2}\left(1+\varepsilon-|x|^{n-2}\varphi^{2-n}(x)\right)^{-2}.
\end{equation}
For every $\varepsilon>0$ problem \eqref{eq7} is well-posed. For the functional
$$
J^\varepsilon_{u_0}(f)=\frac{1}{2}\int_{(0,T)\times\Omega}u^2(t,x)dtdx+\frac{1}{2}\int_0^T\|f\|_{H^{-1}(\Omega)}dt
$$
where $f$ is localized in $\omega\Subset\Omega\backslash\{0\}$ and $u$ is solution of \eqref{eq7} we have the following result.
\begin{theorem}
\label{th2}
Suppose $\mu>\left(\frac{n-2}{2}\right)^2$, $\omega\Subset\Omega\backslash\{0\}$, $n\geq3$ and $f$ is localized in $\omega$. Then there is no constant $C_0$ such that for all $\varepsilon>0$ and  $u_0\in L^2(\Omega)$
$$
\inf_{f\in D_1(f)}J^\varepsilon_{u_0}(f)\leq C_0\|u_0\|^2_{L^2(\Omega)}
$$
where $D_1(f)=\left\{f\in L^2\left((0,T);H^{-1}(\Omega)\right)\right\}$.
\end{theorem}
 In order to prove Theorem \ref{th2} we need the following spectral estimates for the operator
\begin{equation}
\label{eq11}
\left\{\begin{array}{ll}
&L^\varepsilon(u)\equiv -\Delta u-\mu \Psi_\varepsilon u, \ \ \hbox{ in } \Omega,
\\
&u=0 \ \ \hbox{ on }  \partial\Omega.
\end{array}\right.
\end{equation}
Let $\lambda_1^\varepsilon$ be the first eigenvalue of (\ref{eq11}), $\phi_1^\varepsilon(x)$ be the corresponding first eigenfunction, $\|\phi_1^\varepsilon(x)\|_{L^2(\Omega)}=1$, i.e.,
\begin{equation}\left\{\begin{array}{ll}
\label{eq12}
&-\Delta\phi^\varepsilon_1-\mu \Psi_\varepsilon(x)\phi^\varepsilon_1=\lambda^\varepsilon_1\phi^\varepsilon_1, x\in\Omega,
\\
&\phi^\varepsilon_1=0, x\in\partial\Omega,
\end{array}\right.
\end{equation}
and $\Psi_\varepsilon(x)$ is defined in (\ref{eq8}).
\begin{proposition}
\label{prop1}
Suppose $\mu>\left(\frac{n-2}{2}\right)^2$, $n\geq3$. Then we have
\begin{equation}
\label{eq13}
\lim_{\varepsilon\rightarrow0}\lambda_1^\varepsilon=-\infty
\end{equation}
and for all $\rho>0, \delta>0, \rho<(1-\delta)\varphi(x)$, $U_{\rho,\delta}=\{x:\rho<|x|<(1-\delta)\varphi(x)\}$,
\begin{equation}
\label{eq14}
\lim_{\varepsilon\rightarrow0}\|\phi_1^\varepsilon\|_{H^1(U_{\rho,\delta})}=0.
\end{equation}
\end{proposition}
\begin{proof}
We assume by contradiction that $\lambda_1^\varepsilon$ is bounded from below with a constant $C_1$. Then from the Reyleigh identity it follows that
\begin{equation}
\label{eq15}
\mu\int_\Omega \Psi_\varepsilon(x)u^2dx\leq\int_\Omega|\nabla u|^2dx-C_1\int_\Omega u^2dx,
\end{equation}
for every $u\in H^1_0(\Omega)$. For every $a\geq1$ we define $u_a=a^nu(ax)$ so that (\ref{eq15}) becomes
\begin{equation}
\label{eq16}
\begin{array}{ll}
&\mu a^{2n}\int_\Omega u^2(ax)\Psi_\varepsilon(x)dx
\leq a^{2n+2}\int_\Omega |\nabla u(ax)|^2dx-C_1a^{2n}\int_\Omega u^2(ax)dx
\end{array}
\end{equation}
After the limit $\varepsilon\rightarrow0$ and then the change of variables $ax=y$ we get from (\ref{eq16})
\begin{equation}
\label{eq17}\begin{array}{ll}
&\mu a^{2}\int_{|y|<a\varphi(y)} u^2(y)|y|^{-2}\left(1-|y|^{n-2}(\varphi(y)a)^{2-n}\right)^{-2}dy
\\
&\leq a^2\int_{|y|<a\varphi(y)}|\nabla u(y)|^2dy-C_1\int_{|y|<a\varphi(y)}u^2(y)dy.
\end{array}
\end{equation}
For every $u\in C_0^\infty(\Omega)$ and for every fixed $y$ we have
$
u^2(y)|y|^{-2}\left(1-|y|^{n-2}(\varphi(y)a)^{2-n}\right)^{-2}\rightarrow_{a\rightarrow\infty}u^2(y)|y|^{-2}.
$
Hence $\hbox{supp }u\subset \{|y|<\delta\varphi(y)\}$ for some $\delta\in (0,1)$ and for every $a\geq1$ and every $y$ it follows
$$
u^2(y)|y|^{-2}\left(1-|y|^{n-2}(\varphi(y)a)^{2-n}\right)^{-2}\leq u^2(y)|y|^{-2}\left(1-\delta^{n-2}\right)^{-2}<\infty.
$$
Thus after the limit $a\rightarrow\infty$ in (\ref{eq17}) we get
\begin{equation}
\label{eq18}
\int_\Omega\nabla u(y)|^2dy\geq\mu\int_\Omega\frac{u^2(y)dy}{|y|^2},
\end{equation}
for every $u\in C_0^\infty(\Omega)$ and by the continuity (\ref{eq18}) holds for every $u\in H^1_0(\Omega)$.

Since $\mu>\left(\frac{n-2}{2}\right)^2$, inequality (\ref{eq18})  contradict the Hardy inequality \eqref{eq200} with optimal constant $\left(\frac{n-2}{2}\right)^2$ and
\eqref{eq13} is proved.

In order to prove (\ref{eq14}) let us consider a non-negative smooth function $\eta(x)$, such that $\|\eta\|_{L^\infty(R^n)}\leq1$
 and
 $$\left\{\begin{array}{ll}
 & \eta(x)=1, \ \ \hbox{for } x\in\{\rho<|x|<(1-\delta)\varphi(x)\}
 \\
 &\eta(x)=0, \ \ \hbox{in } \{|x|<\frac{\rho}{2}\}\cup \left\{\left(1-\frac{\delta}{2}\right)\varphi(x)<|x|<\varphi(x)\right\}.
 \end{array}\right.
 $$
Multiplying (\ref{eq12}) with $\eta\phi^\varepsilon_1(x)$ and integrating in $\Omega$ we get
\begin{equation}
\label{eq19}\begin{array}{ll}
&\int_\Omega\eta|\nabla\phi^\varepsilon_1|^2dx-\lambda^\varepsilon_1\int_\Omega\eta(\phi^\varepsilon_1)^2dx
=\mu\int_\Omega\eta \Psi_\varepsilon(x)(\phi^\varepsilon_1)^2dx+\frac{1}{2}\int_\Omega(\phi^\varepsilon_1)^2\Delta\eta dx.
\end{array}
\end{equation}
From (\ref{eq19}), the choice of $\eta$ and the unit $L^2$ norm of $\phi^\varepsilon_1$ it follows that
$$
-\lambda^\varepsilon_1\int_\Omega\eta(\phi^\varepsilon_1)^2dx\leq4\mu\rho^{-2}\left[1-\left(1-\frac{\delta}{2}\right)^{n-2}\right]^{-2}+\frac{1}{2}\|\Delta\eta\|_{L^\infty(\Omega)}.
$$
By means of (\ref{eq13}) we get $\lim_{\varepsilon\rightarrow0}\int_\Omega\eta(\phi^\varepsilon_1)^2dx=0$ and hence
\begin{equation}
\label{eq21}
\lim_{\varepsilon\rightarrow0}\int_{U_{\rho,\delta}}(\phi^\varepsilon_1)^2dx=0, \hbox{for every } U_{\rho,\delta}=\{\rho<|x|<(1-\delta)\varphi(x)\}.
\end{equation}
Now using (\ref{eq19}), (\ref{eq21}) for $U_{\frac{\rho}{2},\frac{\delta}{2}}=\left\{\frac{\rho}{2}<|x|<(1-\frac{\delta}{2})\varphi(x)\right\}$ it follows that
$$\begin{array}{lll}
&&\int_{U_{\rho,\delta}}|\nabla\phi^\varepsilon_1|^2dx\leq\int_\Omega\eta|\nabla\phi^\varepsilon_1|^2dx
\\
&&\leq\left[4\mu\rho^{-2}\left(1-\left(1-\frac{\delta}{2}\right)^{n-2}\right)^{-2}+\frac{1}{2}\|\Delta\eta\|_{L^\infty(\Omega)}\right]
\int_{U_{\frac{\rho}{2},\frac{\delta}{2}}}\left(\phi^\varepsilon_1\right)^2dx\rightarrow_{\varepsilon\rightarrow0}0,
\end{array}
$$
which proves (\ref{eq14}).
\end{proof}
\begin{proof}[Proof of Theorem \ref{th2}]
Due to (\ref{eq13}) we fixe $\varepsilon>0$ sufficiently small such that $\lambda_1^\varepsilon<0$ and choose $u_0=\phi^\varepsilon_1$, $\|\phi^\varepsilon_1\|_{L^2(\Omega)}=1$ where $\phi^\varepsilon_1$ is the first eigenfunction of (\ref{eq11}). Let us consider the functions
$
a(t)=\int_\Omega u(t,x)\phi^\varepsilon_1dx, b(t)=\langle f, \phi^\varepsilon_1\rangle_{L^2(\Omega)}.
$
Simple computations give us
$$
\begin{array}{ll}
&a'(t)=\int_\Omega\phi^\varepsilon_1\left(\Delta u+\mu \Psi_\varepsilon(x)u+f\right)dx
=\int_\Omega\left(\Delta\phi^\varepsilon_1+\mu \Psi_\varepsilon(x)\phi^\varepsilon_1(x)\right)udx+\int_\Omega f\phi_1^\varepsilon dx
\\
&=-\lambda^\varepsilon_1\int_\Omega\phi^\varepsilon_1udx+\int_\Omega f\phi_1^\varepsilon dx=-\lambda^\varepsilon_1a(t)+b(t).
\end{array}
$$
So $a(t)$ satisfies the problem
$$
a'(t)+\lambda^\varepsilon_1a(t)=b(t), \ \ a(0)=1.
$$
Hence
$
a(t)=e^{-\lambda^\varepsilon_1t}+\int_0^te^{-\lambda^\varepsilon_1(t-s)}b(s)ds
$
and
\begin{equation}
\label{eq25}\begin{array}{ll}
&\int_{(0,T)\times\Omega}u^2(t,x)dxdt\geq\int_0^Ta^2(t)dt
\\
&\geq \frac{1}{2}\int_0^Te^{-2\lambda^\varepsilon_1t}dt-\int_0^T\left(\int_0^te^{-\lambda^\varepsilon_1(t-s)}b(s)ds\right)^2dt
\\
&\geq-\frac{1}{4\lambda^\varepsilon_1}\left(e^{-2\lambda^\varepsilon_1T}-1\right)+\frac{1}{2\lambda^\varepsilon_1}\int_0^T\left(e^{-2\lambda^\varepsilon_1t}-1\right)\int_0^tb^2(s)dsdt
\\
&\geq-\frac{1}{4\lambda^\varepsilon_1}\left(e^{-2\lambda^\varepsilon_1T}-1\right)+\frac{1}{2\lambda^\varepsilon_1}\int_0^Te^{-2\lambda^\varepsilon_1t}dt\int_0^Tb^2(s)ds
\\
&=-\frac{1}{4\lambda^\varepsilon_1}\left(e^{-2\lambda^\varepsilon_1T}-1\right)-\frac{e^{-2\lambda^\varepsilon_1T}-1}{4(\lambda^\varepsilon_1)^2}\int_0^Tb^2(s)ds.
\end{array}
\end{equation}
Since
$
b^2(t)\leq\|f\|_{H^{-1}(\omega)}\|\phi_1^\varepsilon\|_{H^1(\omega)}, \ \ \omega\Subset\Omega\backslash\{0\}
$
we get from (\ref{eq25})
$$
\begin{array}{ll}
&-\frac{e^{-2\lambda^\varepsilon_1T}-1}{4\lambda^\varepsilon_1}\leq\int_{(0,T)\times\Omega}u^2(t,x)dxdt
+\frac{e^{-2\lambda^\varepsilon_1T}-1}{4(\lambda^\varepsilon_1)^2}\|\phi^\varepsilon_1\|_{H^1(\omega)}\int_0^T\|f(t,.)\|_{H^{-1}(\Omega)}dt.
\end{array}
$$
Therefore either
$$
-\frac{e^{-2\lambda^\varepsilon_1T}-1)}{8\lambda^\varepsilon_1}
\leq\frac{e^{-2\lambda^\varepsilon_1T}-1}{4(\lambda^\varepsilon_1)^2}\|\phi^\varepsilon_1\|_{H^1(\omega)}\int_0^T\|f(t,.)\|_{H^{-1}(\Omega)}dt,
$$
or
$$
-\frac{e^{-2\lambda^\varepsilon_1T}-1}{8\lambda^\varepsilon_1}\leq\int_{(0,T)\times\Omega}u^2(t,x)dxdt.
$$
In any case we have for every $f$ localized in $\omega\Subset\Omega\backslash\{0\}$ the estimate
$$
J^\varepsilon_{u_0}(f)\geq\inf\left\{-\frac{e^{-2\lambda^\varepsilon_1T}-1}{16\lambda^\varepsilon_1}, \frac{-\lambda^\varepsilon_1}{4\|\phi^\varepsilon_1\|_{H^1(\omega)}}\right\}
$$
holds.

From Proposition \ref{prop1}, if $\omega\subset U_{\rho,\delta}=\{\rho<|x|<(1-\delta)\varphi(x)\}$  for some positive constants $\rho, \delta$, it follows that
$$
\lim_{\varepsilon\rightarrow0}\|\phi^\varepsilon_1\|_{U_{\rho,\delta}}=0,
$$
and hence $\lim_{\varepsilon\rightarrow0}J^\varepsilon_{u_0}(f)=\infty$ which proves  Theorem \ref{th2}.
\end{proof}
\textbf{Acknowledgement}
\small{The work is  partially supported by   the Grant No BG05M2OP001--1.001--0003, financed by the Science and Education for Smart Growth Operational Program (2014-2020) in Bulgaria and co-financed by the European Union through the European Structural and Investment Funds.}

%\bibliography{BiblioHardyM}

\begin{thebibliography}{20}
\providecommand{\natexlab}[1]{#1}
\providecommand{\url}[1]{\texttt{#1}}
\expandafter\ifx\csname urlstyle\endcsname\relax
  \providecommand{\doi}[1]{doi: #1}\else
  \providecommand{\doi}{doi: \begingroup \urlstyle{rm}\Url}\fi

\bibitem[Baras and Goldstein(1984)]{BG84}
P.~Baras and J.~A. Goldstein.
\newblock The heat equation with singular potential.
\newblock \emph{Trans. Am. Math. Soc.}, 284:\penalty0 121--139, 1984.

\bibitem[Cabr\'e and Martel(1999)]{CM99}
X.~Cabr\'e and Y.~Martel.
\newblock Existance versus instantaneous blowup for linear heat equations with
  singular potentials.
\newblock \emph{C. R. Acad. Sci. Paris Sér. I Math.}, 329:\penalty0 973--978,
  1999.

\bibitem[de~Castro(2004)]{Ca04}
A.S. de~Castro.
\newblock Bound states of the \textsc{D}irac equation for a class of effective
  quadratic plus inversely quadratic potentialss.
\newblock \emph{Ann. Physics}, 311:\penalty0 170--181, 2004.

\bibitem[Peral and Vazquez(1995)]{PV95}
I.~Peral and J.~Vazquez.
\newblock On the stability or instability of the singular solution of the
  semilinear heat equation with exponential reaction term.
\newblock \emph{Arch. Ration. Mech. Anal.}, 129:\penalty0 201--224, 1995.

\bibitem[Levy-Leblond(1967)]{Le67}
J.~M. Levy-Leblond.
\newblock Electron capture by polar molecules.
\newblock \emph{Phys. Rev.}, 153\penalty0 (1):\penalty0 1--4, 1967.

\bibitem[Berestycki and Esteban(1997)]{BE97}
H.~Berestycki and M.~J. Esteban.
\newblock Existence and bifurcation of solutions for an elliptic degenerate
  problem.
\newblock \emph{J. Diff. Eq.}, 134\penalty0 (1):\penalty0 1--25, 1997.

\bibitem[Giri et~al.(2008)Giri, Gupta, Meljanac, and Samsarov]{GGMS08}
P.~R. Giri, K.~S. Gupta, S.~Meljanac, and A.~Samsarov.
\newblock Electron capture and scaling anomaly in polar molecules.
\newblock \emph{Phys. Lett. A}, 372\penalty0 (17):\penalty0 2967–--2970, 2008.

\bibitem[Ansini and Giacomelli(2004)]{AG04}
L.~Ansini and L.~Giacomelli.
\newblock Doubly nonlinear thin-film equation in one space dimension.
\newblock \emph{Arch. Ration. Mech. Anal.}, 173:\penalty0 89--–131, 2004.

\bibitem[Azorero and Alonso(1998)]{AA98}
J.~P.~G. Azorero and I.~P. Alonso.
\newblock Hardy inequalities and some critical elliptic and parabolic problems.
\newblock \emph{J. Diff. Eq.}, 144:\penalty0 441--–476, 1998.

\bibitem[V\'azquez and Zuazua(2000)]{VZ00}
J.~V\'azquez and E.~Zuazua.
\newblock The \textsc{H}ardy inequality and the asymptotic behaviour of the
  heat equation with an inverse--square potential.
\newblock \emph{J. Funct. Anal.}, 173:\penalty0 103--153, 2000.

\bibitem[Goldstein et~al.(2012)Goldstein, Goldstein, and Rhandi]{GGR12}
G.~R. Goldstein, J.~A. Goldstein, and A.~Rhandi.
\newblock Weighted \textsc{H}ardy’s inequality and the \textsc{K}olmogorov
  equation perturbed by an inverse-square potential.
\newblock \emph{Appl. Anal.}, 91\penalty0 (11):\penalty0 2057–--2071, 2012.

\bibitem[Junqiang et~al.(2012)Junqiang, Yongda, and Pengcheng]{JYP12}
H.~Junqiang, W.~Yongda, and N.~Pengcheng.
\newblock Existence of solutions to the parabolic equation with a singular
  potential of the \textsc{S}obolev-\textsc{H}ardy type.
\newblock \emph{Acta Math. Sci.}, 32B\penalty0 (5):\penalty0 1901--1918, 2012.

\bibitem[Vancostenoble(2011)]{Va11}
J.~Vancostenoble.
\newblock Lipschitz stability in inverse source problems for singular parabolic
  equations.
\newblock \emph{Comm. Part. Diff. Eq.}, 36\penalty0 (8):\penalty0 1287--1317,
  2011.

\bibitem[Ervedoza(2008)]{Er08}
S.~Ervedoza.
\newblock Control and stabilization properties for a singular heat equation
  with an inverse-square potential.
\newblock \emph{Comm. Part. Diff. Eq.}, 33\penalty0 (11):\penalty0 1996–--2019,
  2008.

\bibitem[Vancostenoble and Zuazua(2008)]{VZ08}
J.~Vancostenoble and E.~Zuazua.
\newblock Null controllability for the heat equation with singular
  inverse-square potentials.
\newblock \emph{J. Funct. Anal.}, 254\penalty0 (7):\penalty0 1864--1902, 2008.

\bibitem[Cazacu(2012)]{Ca12}
C.~Cazacu.
\newblock Schr\"odinger operators with boundary singularities: \textsc{H}ardy
  inequality, \textsc{P}ohozaev identity and controllability results.
\newblock \emph{J. Funct. Anal.}, 263:\penalty0 3741--3783, 2012.

\bibitem[Biccari and Zuazua(2016)]{BZ16}
U.~Biccari and E.~Zuazua.
\newblock Null controllability for a heat equation with a singular
  inverse-square potential involving the distance to theboundary function.
\newblock \emph{J. Diff. Eq.}, 261:\penalty0 2809--2853, 2016.

\bibitem[Fabricant et~al.(2013)Fabricant, Kutev, and Rangelov]{FKR13}
A.~Fabricant, N.~Kutev, and T.~Rangelov.
\newblock Hardy-type inequality with double singular kernels.
\newblock \emph{Centr. Eur. J. Math.}, 11\penalty0 (9):\penalty0 1689--1697,
  2013.

\bibitem[Boccardo and Murat(1992)]{BM92}
L.~Boccardo and F.~Murat.
\newblock Almost everywere convergence of the gradients of solutions to
  elliptic and parabolic equations.
\newblock \emph{Nonlinear Anal.}, 19\penalty0 (6):\penalty0 437--477, 1992.

\bibitem[Hardy et~al.(1952)Hardy, Polya, and Littlewood]{HPL52}
G.~Hardy, G.~Polya, and J.~Littlewood.
\newblock \emph{Inequalities}.
\newblock Cambridge University Press, Cambridge, 1952.

\end{thebibliography}
%\bibliographystyle{unsrtnat}

\begin{flushleft}
Institute of Mathematics and  Informatics,\\ Bulgarian Academy
of Sciences \\ Acad. G. Bonchev str.,bl. 8\\
Sofia 1113, Bulgaria, \\ E-mail: kutev@math.bas.bg; rangelov@math.bas.bg,
\end{flushleft}

\end{document}